\documentclass{amsart}

\usepackage[english]{babel}
\usepackage[utf8]{inputenc}
\usepackage{amsmath}
\usepackage{url}
\usepackage{breakurl}
\usepackage{mathtools}
\usepackage{graphicx}
\usepackage{amssymb}
\usepackage{amsfonts}
\usepackage{amsthm}
\usepackage[all]{xy}
\usepackage{enumerate}
\usepackage{subfig}
\usepackage{algpseudocode}
\usepackage{color}
\usepackage{enumerate}
\usepackage{longtable}
\usepackage{array}
\usepackage{float}
\usepackage{mathdots}
\usepackage{yhmath}
\usepackage{MnSymbol}
\usepackage{footnote}
\restylefloat{table}
\usepackage{stmaryrd}
\usepackage{hyperref}
\usepackage{fullpage}
\usepackage{multirow}

\newcommand{\mainbound}{\frac{1}{8} B^{10}}
\newcommand{\RR}{\mathbb{R}}
\newcommand{\CC}{\mathbb{C}}
\newcommand{\QQ}{\mathbb{Q}}

\newcommand{\ZZ}{\mathbb{Z}}
\newcommand{\End}{\mathrm{End}}
\newcommand{\Tr}{\mathrm{Tr}}
\newcommand{\N}{\mathrm{N}}
\newcommand{\Jac}{\mathrm{Jac}}
\newcommand{\Jacof}[1]{\Jac(#1)}
\newcommand{\Hom}{\mathrm{Hom}}
\newcommand{\Real}{\mathrm{Re}}
\newcommand{\im}{\mathrm{Im}}
\newcommand{\vol}{\mathrm{vol}}
\newcommand{\covol}{\mathrm{covol}}
\newcommand{\oJ}{\overline{J}}

\newcommand{\resfield}{\mathfrak{K}}
\newcommand{\Mat}[1]{\mathrm{Mat}_{#1\times #1}} 
\newcommand{\OMp}{\mathcal{O}_{\mathfrak{p}}}
\newcommand{\Bone}{\mathcal{B}_1} 

\DeclareMathOperator{\Spec}{\mathrm{Spec}}

\newcommand{\calC}{{\mathcal C}}

\newcommand{\OO}{{\mathcal O}}
\newcommand{\frp}{{\mathfrak p}}

\newtheorem{theorem}{Theorem}[section]
\newtheorem{definition}[theorem]{Definition}
\newtheorem{conjecture}[theorem]{Conjecture}
\newtheorem{proposition}[theorem]{Proposition}
\newtheorem{corollary}[theorem]{Corollary}

\newtheorem{lemma}[theorem]{Lemma}

\newtheorem{remark}[theorem]{Remark}

\begin{document}
\title{A bound
on the primes of bad reduction for CM curves of genus~$3$}

\author{P{\i}nar K{\i}l{\i}\c{c}er}
\address{P{\i}nar K{\i}l{\i}\c{c}er, Bernoulli Institute for Mathematics, Computer  Science and AI, Nijenborgh 9, 9747 AG Groningen, The Netherlands} \email{p.kilicer@rug.nl}
\thanks{}

\author{Kristin Lauter}
\address{Kristin Lauter, Microsoft Research, Cryptography,
One Microsoft Way,
Redmond, WA, 98052 USA} \email{klauter@microsoft.com}
\thanks{}

\author{Elisa Lorenzo Garc\'ia}
\address{Elisa Lorenzo Garc\'ia, IRMAR, Université de Rennes 1,
Campus de Beaulieu,
35042 Rennes cedex, France}
\email{elisa.lorenzogarcia@univ-rennes1.fr}
\thanks{}

\author{Rachel Newton}
\address{Rachel Newton, Department of Mathematics and Statistics, 
University of Reading, Whiteknights, PO Box 220,
Reading RG6 6AX,
UK}
\email{r.d.newton@reading.ac.uk}
\thanks{Newton is supported by EPSRC grant EP/S004696/1.
Ozman is supported by Bogazici University Research Fund Grant Number 15B06SUP3 and by the BAGEP award of the Science Academy, 2016.
Streng is supported by NWO Vernieuwingsimpuls.
}

\author{Ekin Ozman}
\address{Ekin Ozman, Bogazici University, Faculty of Arts and Sciences, Mathematics Department, Bebek, Istanbul, 34342, Turkey}
\email{ekin.ozman@boun.edu.tr}

\author{Marco Streng}
\address{Marco Streng, Mathematisch Instituut,
Universiteit Leiden,
PO Box 9512,
2300 RA Leiden,
The Netherlands}
\email{streng@math.leidenuniv.nl}

\begin{abstract}
We give bounds on the primes of geometric
bad reduction for curves of genus three
of primitive CM type in terms of the CM orders.
In the case of elliptic curves,
there are no primes of geometric bad reduction because CM elliptic
curves are CM abelian varieties, which have potential good reduction everywhere.
However,
for genus at least two, the curve can have bad reduction at a prime although the Jacobian has good reduction.
Goren and Lauter gave the first bound in the case of genus two.

In the cases of hyperelliptic and Picard curves, our results
imply bounds on primes appearing in the denominators
of invariants and class polynomials,
which are important for algorithmic construction
of curves with given characteristic polynomials over finite fields.
\end{abstract}

\maketitle

\section{Introduction}
Generating
curves over finite fields with a given number of points on the curve or on its Jacobian is a hard and interesting problem, with valuable applications and connections to number theory.  The case of
elliptic curves, for example, has important applications in cryptography, and current solutions rely on computing Hilbert class polynomials associated to imaginary quadratic fields.  For curves of genus~2, already additional interesting problems arise when trying to compute the analogous class polynomials, the Igusa class polynomials, since the coefficients are not integral as in the
case of genus~1.  This leads to the question of understanding and bounding primes of bad reduction for curves of genus~2 whose Jacobians have complex multiplication (CM), and connections with arithmetic intersection theory (\cite{GorenLauter, LauterViray}).

The case of genus~3 is more complicated than the genus 2 case.
First, an abelian threefold can be non-simple \emph{without} being isogenous to a product of elliptic curves. Second, it is possible for a sextic CM field to have both primitive \emph{and} non-primitive CM types. Third, the rank of the endomorphism algebra can be larger in the genus~3 case than in the genus~2 one. Handling each of these complications requires new ideas.

In this paper, we prove the following result which gives a bound on primes of geometric bad reduction for CM
curves of genus~3 with primitive CM type
(here and in what follows, we say that a curve has CM
if its Jacobian does, and we refer to the CM type of
the Jacobian also as the CM type of the curve).

\begin{theorem}\label{thm:main}
Let $C/M$ be a 
smooth, projective, geometrically irreducible
curve of genus $3$ over a number field~$M$.
Suppose that the Jacobian $\Jacof{C}$ has CM by an order $\mathcal{O}$ 
inside a CM field $K$ of degree $6$
and that the CM type of $\Jacof{C}$ is primitive. 
Let $\mathfrak{p}$ be a prime of $M$ lying over a rational prime $p$
such that $C$
does not have potential good reduction at $\mathfrak{p}$.
Then the following upper bound holds on $p$. For every $\mu\in \mathcal{O}$ with $\mu^2$ totally real
and $K = \QQ(\mu)$, we have $p < \mainbound$ where $B =-\frac{1}{2}\Tr_{K/\QQ}(\mu^2)$.
\end{theorem}

As in the case of genus two~\cite{GorenLauter}, in order to prove Theorem~\ref{thm:main}, we use the fact that bad reduction of $C$ gives an embedding of the CM order $\mathcal{O}$ into the endomorphism ring of the reduced Jacobian such that the Rosati involution induces complex conjugation on $\mathcal{O}$ (see Lemma~\ref{lem:dual}). We show that such an embedding cannot exist for sufficiently large primes. The proof of Theorem \ref{thm:main} is given in Section \ref{sec:proofmain}.

To deal with the new situation where the reduction is
a product of an elliptic curve with an abelian surface with no
natural decomposition, we needed to find a suitable and explicit
decomposition. Just the existence of a decomposition is not enough,
and our first main contribution is to find the `right' decomposition
(Lemma~\ref{lem:s}).

The second main new idea is using the primitivity of the CM type
in the case where there exist non-primitive CM types. For this we
use the reduction of the tangent space in Section~\ref{sec:primitive}.
Primitivity is crucial for our methods, but we do give
the following conjecture in the non-primitive case.
\begin{conjecture}\label{conj}
There is a constant $e\in\RR_{\geq 0}$ such that the following holds.
Let $C/M$ be a 
smooth, projective, geometrically irreducible
curve of genus $g\leq 3$ over a number field~$M$.
Suppose that $C$ has CM (not necessarily of primitive CM type)
by an order $\mathcal{O}$
in a CM field $K$ of degree $2g$.

Let $\mathfrak{p}$ be a prime of $M$ lying over a rational prime $p$
such that $C$
does not have potential good reduction at $\mathfrak{p}$.
Then the following upper bound holds on $p$. For every $\mu\in \mathcal{O}$ with $\mu^2$ totally real
and $K = \QQ(\mu)$, we have $p < B^e$ where $B =-\frac{1}{2}\Tr_{K/\QQ}(\mu^2)$.
\end{conjecture}
\begin{remark}
The case $g=1$ is true even with $e=0$, as CM elliptic curves
have potential good reduction everywhere.
The case of primitive CM types is Goren-Lauter~\cite{GorenLauter}
for $g=2$ and Theorem~\ref{thm:main} for $g=3$.
The case of non-primitive CM types is an open problem
even for $g=2$ as far as we know.

We do have numerical evidence in the case $g=2$.
Br\"oker-Lauter-Streng~\cite[Lemma 6.4, Tables 1 and~2]{Broker-Lauter-Streng} 
give CM hyperelliptic curves
$C_{-3}$, $C_{-6}^{1}$, $C_{-6}^2$,  $C_{-8}$,
$C_{-15}$, $C_{-20}^{i}$, $C_{-20}^{-i}$
as well as explicit CM orders~$\mathcal{O}$,
and each time the denominators of the absolute
Igusa invariants have only small prime factors.
For example, we have $(I_4I_6/I_{10})(C_{-15}) = - 3^{2} \cdot 5^{3} \cdot 79\ /\ 2^{7}$.

A proof in the case where the CM type is non-primitive
cannot use the tangent space in the way we use it
in our proof.
On the other hand, in the case of non-primitive CM types there are
more endomorphisms that one
could use. This is because (for $g\leq 3$) the endomorphism ring
$\End(J_{\overline{M}})$
has rank $2g^2$ over $\ZZ$, whereas in the case
of primitive CM types we have
$\End(J_{\overline{M}})\cong \mathcal{O}$ of rank~$2g$. Here, and throughout, $\overline{M}$ denotes an algebraic closure of $M$.
\end{remark}

The following proposition, which is proven in Section \ref{sec:geom}, turns the bound of Theorem~\ref{thm:main}
into an intrinsic bound, depending only on the discriminants
of the orders involved.
\begin{proposition}\label{prop:geomnumbers}
Let $\mathcal{O}\subset K$ be an order in a sextic CM field.
\begin{enumerate}
\item If $K$ contains no imaginary quadratic subfield,
then there exists $\mu$ as in Theorem~\ref{thm:main}
satisfying $0 < -\frac{1}{2}\Tr_{K/\QQ}(\mu^2) \leq (\frac{6}{\pi})^{2/3}|\Delta(\OO)|^{1/3}$, where $\Delta(\OO)$ is the discriminant of the order $\OO$.
\item If $K$ contains an imaginary quadratic subfield~$K_1$,
let $K_+$ be the totally real cubic subfield and let \mbox{$\mathcal{O}_i = K_i\cap \mathcal{O}$} where $i\in\{1,+\}$.
Then there exists $\mu$ as in Theorem~\ref{thm:main}
with $0 < -\frac{1}{2}\Tr_{K/\mathbb{Q}}(\mu^2) \leq |\Delta(\OO_1)|(1+2\sqrt{|\Delta(\OO_+)|})$.
\end{enumerate}
\end{proposition}

Our next result is a consequence of Theorem~\ref{thm:main} in the special cases of hyperelliptic and Picard curves.
A hyperelliptic curve of genus $3$ over a subfield $M$ of $\CC$
is a curve with an affine model of the form $C : y^2 = F(x,1)$
such that~$F$ is a separable binary form over $M$
of degree~$8$. 
A hyperelliptic curve invariant of weight~$k$ for genus~$3$ 
is a polynomial $I$ over~$\ZZ$ in the coefficients of~$F$
satisfying $I(F\circ A) = \det(A)^k I(F)$
for all $A\in\mathrm{GL}_2(\CC)$.
For example, the discriminant $\Delta$ of $F$
(not to be confused with $\Delta(\mathcal{O})$)
is an invariant of weight $56$.
Shioda~\cite{Shioda} gives a set of invariants
that uniquely determines the isomorphism class of $C$ over $\CC$. 

A Picard curve of genus $3$ over a field $M$ of characteristic $0$
is a smooth plane projective curve given by an affine model $C : y^3 = f(x)$ such that
$f$ is a monic separable polynomial over $M$
of degree $4$.
Such a curve can be written
as follows (uniquely up to scalings $(x,y) \mapsto (u^3x,u^4 y)$ with $u\in M^*$, 
which change $a_l$ into $u^{3l} a_l$):
\begin{equation}
\label{eq:picard}
y^3 = f(x)=x^4 + a_{2}x^2 + a_{3} x + a_{4}.
\end{equation}
We define the ring of invariants to be the
graded ring generated over $\ZZ[\frac{1}{3}]$
by the symbols $a_2$, $a_3$ and $a_4$ of respective weights $2$, $3$, $4$.
It contains the discriminant $\Delta$ of
$f(x)$,
which is an invariant of weight $12$.

The following consequence of Theorem~\ref{thm:main}
is derived in Section~\ref{sec:invariants}.
\begin{theorem}\label{thm:invariantdenom}
Let $C/M$ be a hyperelliptic (respectively Picard)
curve of genus $3$ over a number field~$M$.
Suppose that $C$ has CM by an order $\mathcal{O}$
inside a CM field $K$ of degree $6$
and that the CM type of $C$ is primitive.
Let $l\in\ZZ_{>0}$ and let
$j = u/\Delta^l$ be a quotient of invariants of hyperelliptic (respectively Picard)
curves, such that the numerator $u$ has weight $56l$ (respectively $12l$).
 Let $\mathfrak{p}$ be a prime over a prime number $p$ such that 
$\mathrm{ord}_{\mathfrak{p}}(j(C))< 0$.
Then $p$ satisfies the bound of Theorem~\ref{thm:main}.
\end{theorem}

\begin{remark}\label{rem:picardcase} In the Picard curve case, subsequent work 
of K{\i}l{\i}\c{c}er, Lorenzo Garc\'ia and Streng~\cite{KLS} using the results of Section~\ref{sec:primitive}
gives a much stronger analogue of Theorem~\ref{thm:invariantdenom}
for the alternative invariants
$a_2a_4/a_3^2
$ and $a_2^3/a_3^2
$.
\end{remark}

In the Picard curve case, we define $j_1 = a_2^{6}/\Delta$, $j_2 = a_2^3a_3^2/\Delta$, $j_3 = a_2^4a_4/\Delta$, $j_4 = a_3^4/\Delta$, $j_5=a_4^3/\Delta$, $j_6 = a_2a_3^2a_4/\Delta$ and $j_7 = a_2^2a_4^2/\Delta$.
Over an algebraic closure $\overline{M}$ of $M$, any Picard curve has a model in one of the following forms:
\begin{align*}
y^3 &= x^4 + Ax^2 + Ax + B, & & A = j_1^{\phantom{1}}j_2^{-1},\ B = j_1^{\phantom{1}}j_2^{-2}j_3^{\phantom{1}}=j_3^{\phantom{1}}j_4^{-1} & & \mbox{if $j_2\not=0$},\\
y^3 &= x^4 + Ax^2 + Bx + B, & & A = j_6^{\phantom{1}}j_5^{-1},\ B = j_4^{\phantom{1}}j_5^{-1} & & \mbox{if $j_4j_5\not=0$},\\
y^3 &= x^4 + x^2 + A, & & A = j_3^{\phantom{1}}j_1^{-1} & & \mbox{if $j_1\not=0$, $j_2=0$},\\
y^3 &= x^4 + x, & &  & & \mbox{if $j_1=j_5=0$},\\
y^3 &= x^4 + 1, & &  & & \mbox{if $j_1=j_4=0$}.
\end{align*}

We use the same notation $j_l$ also in the hyperelliptic
case, but there we take it to mean 
the following quotients
of Shioda invariants
appearing in Weng~\cite[(5)]{Weng}:
$j_1 = I_2^7/\Delta$,
$j_3 = I_2^5I_4/\Delta$, \mbox{$j_5 = I_2^4I_6/\Delta$}, $j_7=I_2^3I_8/\Delta$
and $j_9 = I_2^2 I_{10}/\Delta$.
Note that these
invariants
satisfy the hypothesis of Theorem~\ref{thm:invariantdenom}. 

Now suppose that $K$ is a sextic CM field
containing a primitive $4$th root of unity
and consider invariants of hyperelliptic curves.
Alternatively, let $K$ be a sextic CM field
containing a primitive $3$rd root of unity
and consider invariants of Picard curves.
Let $j=u/\Delta^l$ and $j'=u'/\Delta^l$ be quotients of invariants
of hyperelliptic (respectively
Picard) curves, such that the numerators $u$ and $u'$ have weight 56l
(respectively 12l).
We define the class polynomials $H_{K,j}$ and $\widehat{H}_{K,j,j'}$
by

$$H_{K,j} = \prod_{C} (X-j(C)), \;\; \;\; \widehat{H}_{K,j,j'} = \sum_{C} j'(C) \prod_{D\not\cong C} (X-j(D)) \in \CC[X]$$
where the products and sum range over
isomorphism classes of curves $C$ and $D$ over $\CC$ with
CM by $\mathcal{O}_K$ of primitive CM type,
which are indeed hyperelliptic (resp.~Picard) by
Weng \cite[Theorem 4.5]{Weng} (resp. Koike-Weng~\cite[Lemma 1]{KoikeWeng}).
The polynomial $\widehat{H}_{K,j,j'}$ is the modified Lagrange interpolation of the roots of $H_{j'}$ introduced in
\cite[Section~3]{GHKRW}.
These polynomials have rational coefficients
as they are fixed by $\mathrm{Aut}(\CC)$.
Moreover, the polynomials
$H_{j_l}$ and $\widehat{H}_{j_1,j_l}$, where $l$ ranges over $\{3, 5, 7, 9\}$ in the
hyperelliptic case and over $\{2,3\}$ in the Picard case, 
can be used for the CM method for constructing
curves over finite fields.
See \cite[Section~3]{GHKRW} 
as well as \cite{Weng} 
(resp.~\cite{KoikeWeng}) for how to use these polynomials.

The polynomials
$H_{K, j}$ and $\widehat{H}_{K,j,j'}$ can be approximated
using the methods of
Weng~\cite{Weng} and Balakrishnan-Ionica-Lauter-Vincent~\cite{BILV}
in the hyperelliptic case and the methods of Koike-Weng~\cite{KoikeWeng}
and
Lario-Somoza~\cite{LarioSomoza}
in the Picard case.
The (rational) coefficients
of the polynomials
can then be recognized from such approximations
using continued fractions or the LLL algorithm.
However, to be absolutely sure of the coefficients, 
one would need a bound on the denominators.
We view the following result as a first step
towards obtaining such a bound. 
It is an immediate consequence of
Theorem~\ref{thm:invariantdenom}.
\begin{theorem}\label{thm:classpoly}
Let $K$ be a sextic CM field containing a primitive $4$th root 
of unity and let
$p$ be a prime number that divides the denominator
of a class polynomial $H_{j}$ or $\widehat{H}_{j,j'}$
with quotients of hyperelliptic curve invariants $j$ and $j'$
as in Theorem~\ref{thm:invariantdenom}.
Then $p$ satisfies the bound of Theorem~\ref{thm:main}. The statement remains true if one replaces `4th' by `3rd' and `hyperelliptic' by `Picard'.\qed
\end{theorem}

\subsection{Applications, further work and open problems}
\subsubsection*{Sharper upper bounds, and exponents.} We believe that the exponent $10$ in Theorem~\ref{thm:main}
is not optimal. For instance, in \cite{BCLLMNO}, for the special case of reduction to a product of 3 elliptic curves with $K$ not containing any proper CM subfield, one gets an exponent of $6$. 
In the general case, it may be possible
to get smaller exponents using variants of our proof,
for example
with a different choice of isogeny $s$ in Section~\ref{sec:embprb}, or by considering bounds
in Section~\ref{sec:coeffbound}
coming not just from the matrix of $\mu$,
but also from other elements.

We also believe that it is now possible to
combine our proofs with the
techniques of Goren and Lauter~\cite{GorenLauter2}
to get not only a bound on the primes 
in the denominator of Theorems \ref{thm:invariantdenom}
and~\ref{thm:classpoly}, but also a bound on the
valuations at those primes.
Together, these bounds will give a bound
on the denominator itself,
which is required if one wants to prove
that the output of a 
class-polynomial-computing algorithm is correct.
This was done for genus $2$ by Streng~\cite{Streng}.
As in the case of genus~$2$,
the resulting bounds will be so large that the algorithm
is purely theoretical and cannot be run in practice.
However, we view our results as a first step towards
a denominator formula such as that of
Lauter and Viray~\cite{LauterViray},
which is small and explicit enough for yielding
proven-correct CM curves, as shown by
Bouyer and Streng~\cite{BouyerStreng, StrengImplementation}.

\subsubsection*{Denominators for general curves of genus $3$}

Theorem \ref{thm:invariantdenom}
(and hence~\ref{thm:classpoly}) is only for hyperelliptic and Picard curves.
The reason why it
follows
from Theorem~\ref{thm:main} (as shown in Section~\ref{sec:invariants})
is that 
the primes dividing the denominator $\Delta^l(C)$ of $j(C)$
are exactly the primes of bad reduction for~$C$.
In other words, it is because the zero locus of $\Delta$
in the compactification of the moduli space
of hyperelliptic/Picard curves parametrizes only singular curves.
In the case of the moduli space of all curves of genus three,
the locus of bad reduction has codimension greater than $1$, hence is
not the vanishing locus of an invariant.
In particular, no generalization of Theorem~\ref{thm:invariantdenom}
would follow directly from Theorem~\ref{thm:main}
or even from Conjecture~\ref{conj}.

The most direct generalization
of Theorem~\ref{thm:invariantdenom} to
arbitrary primitive CM curves of genus $3$ would have the
discriminant invariant of plane quartics as the denominator
of~$j$.
Numerical experiments of
K{\i}l{\i}{\c{c}}er, Labrande, Lercier, Ritzenthaler, Sijsling
and Streng~\cite{KLLRSS}
suggest that this generalization would be false.

\subsubsection*{Lower bounds} Habegger and Pazuki
\cite[Theorems 1.3 and~4.5(ii)]{HabeggerPazuki} give lower bounds on the denominators
of absolute invariants of CM curves of genus~$2$.
It would be interesting to see whether a similar result is true for hyperelliptic
or Picard curves of genus~$3$.

\subsection{Acknowledgements}
We thank
Irene Bouw, Bas Edixhoven, Everett Howe, Christophe Ritzenthaler and Chia-Fu Yu
for useful discussions and for pointing out some of the references.
We are grateful to the anonymous referees for many helpful suggestions. Part of this work was carried out at
Carl von Ossietzky University of Oldenburg, the
Istanbul Center for Mathematical Sciences,
the Lorentz Center,
the
Max Planck Institute for Mathematics,
UC San Diego, and
the
University of Warwick.

\section{Notation and strategy}

For the reader's convenience, we define some well-known concepts
that are essential for our approach.
By a \emph{curve} over a field~$M$, we mean a smooth, projective,
geometrically irreducible curve over $M$ unless we say otherwise.
\begin{definition}\label{def:hasCM}
Let $\mathcal{O}$ be an order in a \emph{CM field} $K$ of degree
$2g$ over $\QQ$, that is, an imaginary quadratic extension
of a totally real number field.
We say that a curve $C$ of genus $g$ over a number field
$M$ has \emph{complex multiplication} by $\mathcal{O}$
if there exists an embedding $\phi$ of $\mathcal{O}$ into
the endomorphism ring of the Jacobian $\Jacof{C}_{\overline{M}}$ of $C$
over the algebraic closure.
\end{definition}
\begin{definition}\label{def:primitivetype}Let $K$ be as in Definition~\ref{def:hasCM}.
A \emph{complex multiplication type (CM type)} of $K$
is a set of~$g$ non-conjugate embeddings $K \hookrightarrow \CC$.
We say that a CM type is primitive if its restriction to any strict CM subfield of $K$ is not a CM type.
\end{definition}
\begin{definition}\label{def:CMtypeof}
Given $J$ and $\phi$ as in Definition~\ref{def:hasCM} with $\overline{M}\subset \CC$,
we obtain
a CM type by diagonalizing the action of $K$ via $\phi$
on the tangent space of $\Jacof{C}_{\overline{M}}$ at~$0$,
and we call this the \emph{CM type of $C$}.
\end{definition}

Now let $C$ be a curve of genus~$3$ defined over a number field $M$ and such that its Jacobian $J = \Jacof{C}$ has complex multiplication by an order $\mathcal{O}$ of a sextic CM field~$K$. Let us assume that the CM type is primitive.  We fix a totally imaginary generator $\mu\in\mathcal{O}$ of $K$ over~$\QQ$. Thus, $\mu^2$
 is a totally negative element
 of $\mathcal{O}$
 that  generates the totally real subfield $K_+$ of $K$.

Let $\mathfrak{p}\mid p$ be a prime such that $C$ does not have
potential good reduction at $\mathfrak{p}$.
In other words, $\mathfrak{p}$ is a prime of geometric bad reduction for $C$,
in the sense that even after extension
of the base field, the curve
$C$ still has bad reduction at all primes above~$\mathfrak{p}$.
As noted in \cite[Section~4.2]{BCLLMNO}, this is equivalent to the
stable reduction of $C$ being non-smooth,
where this type of reduction is simply called ``bad reduction''.
As~$J$ has complex multiplication, it has potential good reduction
at every prime by a result of Serre and Tate~\cite{SerreTate}.
Without loss of generality of our main results, we extend the
field $M$ so that $C$ has a stable model for the reduction
at $\mathfrak{p}$ and $J$ has
good reduction at $\mathfrak{p}$. Let $\overline{J} = (J\ \mathrm{mod}\ \mathfrak{p})$.

By Corollary $4.3$ in \cite{BCLLMNO}, we know that,
possibly after extending the base field again,
there exists an isomorphism 
$\overline{J}\cong E\times A$ as principally polarized abelian varieties (p.p.a.v.) over the new base field, where $E$ is an elliptic
curve with its natural polarization and $A$ is a principally polarized
abelian surface.
This includes the case where there is an isomorphism
$\overline{J}\cong E_1\times E_2\times E_3$
as p.p.a.v.,
where $A \cong E_2\times E_3$ is a product of elliptic curves.
Let us write $\text{End}(E)=\mathcal{R}$ and $\mathcal{B}=\mathcal{R}\otimes \QQ$. 

We will see that there is an isogeny $s:E^2\rightarrow A$
(which is, in fact, already known
by \cite[Theorem $4.5$]{BCLLMNO}).
Once we fix an isogeny~$s$, there are natural embeddings 
\[\iota:\,\mathcal{O}\overset{\iota_0}{\hookrightarrow} \text{End}(E\times A)\overset{\iota_1}{\hookrightarrow} \text{End}(E^3)\otimes \mathbb{Q}\cong\Mat{3}(\mathcal{B}).\]

Step 1 is to show that for sufficiently large primes~$p$,
the entries of $\iota(\mu^2)$ lie in a
field~$\Bone\subset\mathcal{B}$ of degree $\leq 2$ over~$\QQ$.
This is obvious in the case where $E$ is ordinary, and requires
work in the supersingular case. 
As in Goren-Lauter~\cite{GorenLauter}, we prove this by bounding
the coefficients of $\iota(\mu)$.
The main difficulty here was finding an appropriate isogeny~$s$,
as not every isogeny~$s$ allows us to find bounds.

Step 2 is to show that in the situation of Step~1, 
the field $\Bone$ embeds into $K$ and the CM type
is induced from $\Bone$, which contradicts the primitivity of the CM type. In order to show this, we use the tangent
space of the N\'eron model at the zero section.
No analogue of Step 2 was needed in the case of genus~$2$ because a quartic CM field containing an imaginary quadratic subfield
has no primitive CM types.

The special case of $\overline{J}\cong E_1\times E_2\times E_3$
as p.p.a.v.\ where $K$ does not contain an imaginary
quadratic field is the main result of \cite{BCLLMNO}.

The following bound will be convenient in
the sense that it allows us to formulate
Theorem~\ref{thm:main} and Proposition~\ref{prop:field}
without the need for case distinctions.
\begin{lemma}
\label{lem:Bnot1}
Let $B = -\frac{1}{2}\Tr_{K/\mathbb{Q}}(\mu^2)$. Then $B$ is an integer and $B\geq 2$.
\end{lemma}

\begin{proof}
Recall that $K_+$ denotes the totally real cubic subfield of $K$. Since $\mu^2\in \mathcal{O}\cap K_+$, we have $B\in\mathbb{Z}$. %
Since $K=\mathbb{Q}(\mu)$, the element $\mu^2$ is totally negative and hence $B>0$. Now suppose that $B=1$. Let $a\geq b \geq
c \geq 0$ be such that $-a, -b, -c$ are the images of $\mu^2$ inside $\RR$ under
the three embeddings of $K_+$ into $\RR$. Then $a + b + c = B = 1$, so each of
$a, b, c$ is in the interval $(0,1)$. In particular, we get $\Tr_{K+/\mathbb{Q}}(\mu^4)
= a^2+b^2+c^2 < a + b + c = 1$. As this trace is a non-negative integer, it is zero,
hence $a=b=c=0$, contradiction.
\end{proof}

\section{An embedding problem}
\label{sec:embprb}

Throughout Sections \ref{sec:embprb}, \ref{sec:coeffbound} and \ref{sec:primitive}, we fix a prime $\frp\mid p$ that is of good reduction for $J=\Jacof{C}$
and not of potential good reduction for~$C$.
In particular, possibly after extending the base field,
the reduction
satisfies $\oJ\cong E\times A$ as polarized
abelian varieties for a
principally polarized abelian surface $A$ and an
elliptic curve~$E$.
Let $\mathcal{R} = \End(E)$ and $\mathcal{B}=\mathcal{R}\otimes \QQ$,
which is either a quaternion algebra or an imaginary quadratic field.

 We write $K=\QQ(\mu)$ where $\mu^2\in\mathcal{O}$
 is totally negative and
 generates the totally real subfield $K_+$ of $K$.

\newcommand{\homEA}[1]{\fbox{\raisebox{0ex}[3ex][2ex]{$#1$}}}
\newcommand{\homAE}[1]{\framebox[3em][c]{$#1$}}
\newcommand{\homAA}[1]{\framebox[3em][c]{\raisebox{0ex}[3ex][2ex]{$#1$}}}
\newcommand{\EndEA}[4]{{\setlength\arraycolsep{0pt} \left(\begin{array}{cc} #1 & \homAE{#2} \\ \homEA{#3} & \homAA{#4}\end{array}\right)}}
\newcommand{\homEEA}[2]{{\setlength\arraycolsep{0pt} \left(\begin{array}{cc} \homEA{#1} &\homEA{#2}\end{array}\right)}}
\newcommand{\homEEEtoEA}[6]{{\setlength\arraycolsep{0pt} \left(\begin{array}{ccc}
#1 & #2 & #3 \\
\homEA{#4} &\homEA{#5} &\homEA{#6} \end{array}\right)}}
\newcommand{\xyzw}{\EndEA{x}{y}{z}{w}}

Let $\iota_0:\,\mathcal{O}\hookrightarrow \text{End}(E\times A)$ be the injective ring homomorphism coming from reduction of $J$ at $\frp$ and write
\begin{equation}\label{eq:xyzw}
  \iota_0(\mu)=:\xyzw,
\end{equation}
where we have $x\in \mathcal{R}$, $y\in\text{Hom}(A,E)$, $z\in\text{Hom}(E,A)$ and $w\in \text{End}(A)$;
and the sizes of the boxes reflect the dimensions of the domains and codomains of the
homomorphisms.
We define a homomorphism
$$
s = \homEEA{z}{\!\! wz \!\!} :E\times E \longrightarrow A, \;\;\;
(P,Q)\longmapsto z(P) + wz(Q).
$$
We first quickly eliminate the degenerate case where $s$ is not an isogeny.
\begin{lemma}\label{lem:s}
The homomorphism $s$ is an isogeny.
\end{lemma}
\begin{proof}
We will prove that $z$ is not the zero map,
and that the image $wz(E)$ of $wz$ is not contained
in $z(E)$.
It then follows that the image of $s$ has dimension $2$
and hence $s$ is an isogeny.

Suppose that $z$ is the zero map.
Then \eqref{eq:xyzw} gives that $x\in \mathcal{B}$
    is a root of
    the minimal polynomial of $\mu$, which is irreducible of degree $6$, contradiction.
    Therefore, $z$ is non-zero and
    $z(E)\subset A$ is an elliptic curve.

    Now let $E'\subset A$ be an elliptic curve such that
  $
	    s':E\times E' \rightarrow A$,
	    given by
    	$(Q,R) \mapsto z(Q)+R
 $
    is an isogeny.
It follows that we have an
isogeny
\begin{align*}
	F' = 1 \times s' = \homEEEtoEA{1}{0}{0}{0}{z}{1}
    : E\times E\times E' &\longrightarrow E\times A  \text{\;\;given by\;\;} 
	(P,Q,R) \longmapsto (P, z(Q)+R)
\end{align*}
and hence a
further embedding
$
\iota_1' : \text{End}(E\times A) \rightarrow \text{End}(E\times E\times E')\otimes \mathbb{Q}, \;
f \mapsto (F')^{-1} f F'$.

Let
$\iota' = \iota_1'\circ\iota_0 : \mathcal{O}\rightarrow \text{End}(E\times E\times E')\otimes \mathbb{Q}$. Next, we compute the matrix $\iota'(\mu)$.
The first column is
\begin{equation}\label{eq:firstcolumn}
(F')^{-1}\xyzw F'\left(\begin{array}{c} 1 \\ 0 \\ 0
\end{array}\right) = \homEEEtoEA{1}{0}{0}{0}{z}{1}^{-1}
\left(\begin{array}{c} x \\ \homEA{z}\end{array}\right) = \left(\begin{array}{c} x \\ 1 \\ 0
\end{array}\right).
\end{equation}

Now suppose that $wz(E)$ is contained in $z(E)$.
Then we get an element
$z^{-1}wz\in \mathcal{B}$ and hence
$$
	\iota'(\mu)=\left(\begin{array}{ccc}x & * & * \\ 1 & z^{-1}wz & * \\ 0 & 0 & \delta\end{array}\right)\quad\mbox{for some $\delta\in \End(E')\otimes \QQ$}.
	$$
But then $\delta$ is a root of the minimal polynomial of $\mu$, which is a contradiction,
hence $wz(E)$ is not contained in $z(E)$
and the image of $s$ has dimension~$2$.
\end{proof}
It follows that we have an
isogeny
\begin{align*}
	F = 1 \times s = \homEEEtoEA{1}{0}{0}{0}{z}{\!\! wz \!\!}
   \ \  : \ \  E^3 &\longrightarrow E\times A \text{\;\;given by\;\;} 
	(P,Q,R) \longmapsto (P, s(Q,R))
\end{align*}
and hence a
further embedding
$
\iota_1 : \text{End}(E\times A) \rightarrow \text{End}(E^3)\otimes \mathbb{Q}\cong\Mat{3}(\mathcal{B})
$
given by
$
f \mapsto F^{-1} f F.
$
Let $\iota = \iota_1\circ\iota_0 : \mathcal{O}\hookrightarrow \Mat{3}(\mathcal{B})$. Let $n$ be a positive integer such that $[n]\cdot \ker(s)=0$
(from Lemma~\ref{lem:pol_matrix} onwards we will use a specific~$n$).
Then there exists an isogeny $\tilde{s}:\,A\rightarrow E\times E$ such that $s\cdot\tilde{s}=[n]$. 
\begin{lemma}\label{lem:matrix}
	We have
	$$
	\iota(\mu)=\left(\begin{array}{ccc}x & a & b \\ 1 & 0 & c/n \\ 0 & 1 & d/n\end{array}\right), \text{\;\;where\;\;} x,a,b,c,d\in \mathcal{R}.
	$$
	
\end{lemma}

\begin{proof}
The first column is already computed in \eqref{eq:firstcolumn}, which is also valid with $F$ instead of~$F'$.
For the second column, we compute
$$
F^{-1}\xyzw F\left(\begin{array}{c} 0 \\ 1 \\ 0
\end{array}\right) = \homEEEtoEA{1}{0}{0}{0}{z}{\!\! wz \!\!}^{-1}
\left(\begin{array}{c} * \\ \homEA{\!\! wz \!\!}\end{array}\right) = \left(\begin{array}{c} * \\ 0 \\ 1
\end{array}\right).$$
As $F^{-1} = 1\times \frac{1}{n}\tilde{s}$,
we get that the entries of the first row of $\iota(\mu)$
are in $\mathcal{R}$ and the others are in $\frac{1}{n}\mathcal{R}$.
\end{proof}

\section{Bounds on the coefficients}\label{sec:coeffbound}

Our goal in this section 
is to prove the following.
\begin{proposition}\label{prop:field}
If $p>\mainbound$, then 
the image $\iota(\mathcal{O})$ is
inside
the ring of $3\times 3$ matrices over a field $\Bone\subset \mathcal{B}$
of degree $\leq 2$ over $\QQ$.
\end{proposition}
If $\mathcal{B}$ is a field, then we can take $\mathcal{B}_1=\mathcal{B}$.
So in the proof of Proposition~\ref{prop:field}, we
assume that $E$ is supersingular and $\mathcal{B}$
is a quaternion algebra. Then $\mathcal{B}$ is $B_{p,\infty}$,
the quaternion algebra ramified exactly at $p$ and $\infty$.
Let $\Tr$ and $\N$ denote the \emph{reduced} trace and norm on $\mathcal{B}$,
and let $\cdot^\vee$ denote (quaternion) conjugation,
so for all $x\in \mathcal{B}$, we have $\N(x) = xx^\vee = x^\vee x$,
$\Tr(x) = x + x^\vee$, and $x^2 - \Tr(x) x + \N(x) = 0$.
Note that $\mathcal{B}=B_{p,\infty}$
is a quaternion algebra ramified at infinity,
hence a definite quaternion algebra,
so the norm $\N(x)$ is a non-negative number
and equal to zero if and only if $x=0$.

To prove Proposition~\ref{prop:field}, we use
the following result, which states that
small quaternions commute.
\begin{lemma}[Goren and Lauter]\label{lem:quaternionfield}
Let $\mathcal{R}$ be an order in the quaternion algebra $B_{p,\infty}$
and $x, y\in\mathcal{R}$. If
$\N(x)\N(y) < p/4$, then $x$ and $y$ commute.
\end{lemma}
\begin{proof}
We give the main idea for completeness.
For details, see Lemma 2.1.1 and Corollary 2.1.2 of Goren and Lauter~\cite{GorenLauter}
and the proof of Lemma~9.5 of Streng~\cite{Streng}.

If $x$ and $y$ do not commute, then $1$, $x$, $y$, $xy$ span a $\ZZ$-lattice
$L\subset \mathcal{R}\subset B_{p,\infty}$ of
covolume $\leq 4\N(x)\N(y)$, while $\mathcal{R}$ is contained 
in a maximal order of covolume $p$. This is a contradiction if 
$\N(x)\N(y) < p/4$.
\end{proof}

Recall that $\oJ\cong E\times A$
as principally polarized abelian varieties,
where $A = (A, \lambda_A)$ is a principally polarized abelian
surface. In other words, the natural
polarization on $\oJ$ corresponds
to the product polarization $1\times \lambda_A$.

\begin{lemma} \label{lem:pol_matrix}
    The polarization induced by $1\times \lambda_A$ on $E^3$
    via the isogeny $F$ is 
	$$
	\lambda:=F^{\vee}(1\times\lambda_A)F=\left(\begin{array}{ccc} 1& 0&0\\0&\alpha & \beta \\
	0 & \beta^{\vee} & \gamma\end{array}\right)\quad \text{\;\;for some\;\;} \alpha,\gamma\in\mathbb{Z}_{>0}
	    \text{\;and\;} \beta\in\mathcal{R} \text{\;such that\;} \alpha\gamma-\beta\beta^{\vee}\in\ZZ_{>0}.
	$$
	 Here $F^\vee$ denotes the dual isogeny.  Let $n=\alpha\gamma-\beta\beta^{\vee}\in\ZZ_{>0}$.
    Then we have $GF = [n]$ for some isogeny~$G$, and therefore $[n]\ker(F)=0$.
\end{lemma}

\begin{proof}
The first column and row of $\lambda$ are easy to compute. 
The symmetry (i.e., $\alpha,\gamma\in\ZZ$ and the occurrence of $\beta^\vee$)
is Mumford \cite[(3) on page 190]{Mumford}
(equivalently the first part of Application III on page 208 of loc.~cit.).
The positive-definiteness (which implies $\alpha,\gamma,n > 0$)
is the last paragraph of Application III on page 210 of loc.~cit.).
It is now straightforward to compute
$GF=[n]$ for
	$$
	G = \left(\begin{array}{ccc} n& 0&0\\0&\gamma & -\beta \\
	0 & -\beta^{\vee} & \alpha\end{array}\right)F^{\vee}(1\times\lambda_A).
	$$
It follows that the kernel of $F$ is contained in the kernel of~$[n]$.
\end{proof}
From now on, take $n$ as in Lemma~\ref{lem:pol_matrix}.

\begin{lemma}[{Proposition $4.8$ in \cite{BCLLMNO}}]
\label{lem:dual}
For every $\eta\in K$, the complex conjugate $\overline{\eta}\in K$ satisfies
\[\iota(\overline{\eta}) = \lambda^{-1} \iota(\eta)^{\vee}\lambda,\]
where for a matrix $M$, we use $M^{\vee}$ to denote the transpose
of $M$ with conjugate entries.
\end{lemma}
\begin{proof}
Complex conjugation is the Rosati involution, so
$\iota_0(\overline{\eta}) = (1\times \lambda_A)^{-1}\iota_0(\eta)^{\vee}(1\times \lambda_A)$. Conjugation with $F^{-1}$ now yields exactly the equality in the lemma:
\begin{align*}
\iota(\overline{\eta})
&= F^{-1}(1\times \lambda_A)^{-1}  \iota_0(\eta)^\vee (1\times \lambda_A)F\\
&= (F^{-1}(1\times \lambda_A)^{-1} F^{-\vee}) (F^{\vee} \iota_0(\eta)^\vee F^{-\vee})
(F^{\vee}(1\times \lambda_A)F) \\
&= (F^{\vee}(1\times \lambda_A)F)^{-1} (F^{-1} \iota_0(\eta) F)^{\vee} 
(F^{\vee}(1\times \lambda_A)F) = \lambda^{-1} \iota(\eta)^{\vee}\lambda.\qedhere
\end{align*}
\end{proof}
\noindent For $\eta=\mu$, Lemma~\ref{lem:dual} reads
$-\lambda \iota(\mu)=\iota(\mu)^{\vee}\lambda$, that is,
$$
\left(\begin{array}{ccc}- x & -a & -b\\ -\alpha & -\beta & -\alpha c/n-\beta d/n\\ -\beta^{\vee}& -\gamma & -\beta^{\vee}c/n-\gamma d/n\end{array}\right)=
\left(\begin{array}{ccc} x^{\vee} & \alpha & \beta \\ a^{\vee} &\beta^{\vee}&\gamma\\b^{\vee} & (c^{\vee}/n)\alpha+(d^{\vee}/n)\beta^{\vee}&(c^{\vee}/n)\beta+(d^{\vee}/n)\gamma\end{array}\right).
$$
We conclude
\begin{align}
x^{\vee}&=-x\qquad\mbox{(equivalently $\text{Tr}(x)=0$),}\nonumber\\ 
a&=-\alpha\qquad\mbox{(and we already knew $\alpha\in\ZZ_{>0}$),}\nonumber\\
b&=-\beta=\beta^{\vee}\qquad\mbox{(hence $\Tr(\beta)=0$),}\label{eq:identities}\\
\gamma&=-\alpha c/n-\beta d/n\qquad\mbox{(and we already knew $\gamma\in\ZZ_{>0}$),}\nonumber\\
\Tr(\beta^{\vee}c)+\text{Tr}(\gamma d)&=0.\nonumber
\end{align}

\begin{lemma}[{Lemma~6.12 in \cite{BCLLMNO}}]\label{lem:trace}
For every $\eta\in K$, the trace $\mathrm{Tr}_{K/\QQ}(\eta)$
is equal to the sum of the reduced traces of the three diagonal entries
of $\iota(\eta)\in \Mat{3}(\mathcal{B})$.
\end{lemma}
\begin{proof}
Choose a prime $l\nmid np$. Then $\mathrm{Tr}_{K/\QQ}(\eta)$ equals
the trace of $\eta$ when acting on $T_l(J)\otimes \QQ$,
where $T_l(J)$ is the $l$-adic Tate module of~$J$.
This action is preserved by reduction modulo~$\frp$.
Moreover, the isogeny $F$ induces an isomorphism of $l$-adic Tate modules,
hence $\mathrm{Tr}_{K/\QQ}(\eta)$ equals
the trace of $\iota(\eta)$ when acting on $T_l(E\times E\times E)\otimes \QQ$.
The latter trace is exactly the sum of the traces of the actions
of the diagonal entries of $\iota(\eta)$ on $T_l(E)\otimes \QQ$,
which are the reduced traces.
\end{proof}

\begin{remark}
Lemma~6.12 in \cite{BCLLMNO} follows from a special case of Lemma~\ref{lem:trace} in which $\eta$ is an element of the totally real cubic subfield $K_+$ of $K$ and the diagonal entries of $\iota(\eta)$ are integers.  
\end{remark}

Since both $\text{Tr}_{K/\QQ}(\mu)$ and $\Tr(x)$ are~$0$,
Lemma~\ref{lem:trace} applied to $\mu$ gives 
\begin{equation}\label{eq:traced}
\text{Tr}(d)=0.
\end{equation}

Let $B = -\frac{1}{2}\text{Tr}_{K/\mathbb{Q}}(\mu^2)\in\ZZ_{>0}$.
Then Lemmas \ref{lem:trace} and~\ref{lem:matrix}
give
\begin{align}
B = & -\frac{1}{2}\left(\text{Tr}(x^2)+2\text{Tr}(a)+2\text{Tr}\Bigl(\frac{c}{n}\Bigr)+\text{Tr}\Bigl(\frac{d^2}{n^2}\Bigr)\right). \\
\intertext{On the other hand, the equality \eqref{eq:traced} implies $d^\vee = -d$ hence we have $\text{Tr}(d^2/n^2) = -2\text{N}(d/n)$ as $n\in\ZZ_{>0}$. Similary, by \eqref{eq:identities} we have $\text{Tr}(x^2) = -2\text{N}(x)$. Moreover, the equality $\gamma=-\alpha c/n-\beta d/n$ in \eqref{eq:identities} and the fact that $\gamma$ and $\alpha$ are integers give $\text{Tr}(c/n) = - \text{Tr}(\gamma/\alpha + \beta d/(n\alpha)) = -2\gamma/\alpha-\text{Tr}(\beta d/(n\alpha))$. Therefore, by $a = -\alpha\in\ZZ$ in \eqref{eq:identities}, we get}
B = & \text{N}(x)+2\alpha+2\frac{\gamma}{\alpha}+\text{Tr}\Bigl(\frac{\beta d}{n\alpha}\Bigr)+\text{N}\Bigl(\frac{d}{n}\Bigr).
\end{align}
If we manage to rewrite this as a sum of terms that are all non-negative, then
this bounds the individual terms from above by~$B$.

Note that we recognize the final two terms as terms in the expansion
$$
\text{N}\Bigl(\frac{\beta}{\alpha}+\frac{d^{\vee}}{n}\Bigr)=\frac{\text{N}(\beta)}
{\alpha^2}+\text{Tr}\Bigl(\frac{\beta d}{\alpha n}\Bigr)+\text{N}\Bigl(\frac{d}{n}\Bigr), \text{\; so we get\; } B=\text{N}(x)+2\alpha+2\frac{\gamma}{\alpha}-\frac{\text{N}(\beta)}
{\alpha^2}+\text{N}\Bigl(\frac{\beta}{\alpha}+\frac{d^{\vee}}{n}\Bigr).
$$
Next, by the definition of $n$ in Lemma~\ref{lem:pol_matrix}, we have $n = \alpha\gamma-\text{N}(\beta)$, so $n/\alpha^2 = \gamma/\alpha - \text{N}(\beta)/\alpha^2$,
which again allows us to replace two terms, and get
\begin{equation}\label{eq:trace}
B =
\text{N}(x)+2\alpha+\frac{\gamma}{\alpha}+\frac{n}
{\alpha^2}+\text{N}\Bigl(\frac{\beta}{\alpha}+\frac{d^{\vee}}{n}\Bigr),
\end{equation}
in which finally all terms are non-negative,
as the norm of an element of $\mathcal{B}_{p,\infty}$
is non-negative. We immediately
get that each of the individual terms is at most $B$.
So e.g.,
$\N(x)\leq B,\; 2\alpha \leq B,\; \gamma/\alpha \leq B.$
Hence we obtain
\begin{equation}
\label{eq:boundbeta}
\text{N}(\beta)/\alpha^2 = \frac{\alpha\gamma-n}{\alpha^2} \leq \gamma/\alpha \leq B.
\end{equation}

In order to bound $\text{N}(d)$, we use
the following well-known (in)equalities.
\begin{lemma}[Parallelogram law]
For all $e,f\in \mathcal{B}$, we have $\text{N}(e+f)+\text{N}(e-f) = 2(\text{N}(e)+\text{N}(f)).$
\end{lemma}
\begin{proof}
By writing it out, the cross terms cancel on the left-hand side and do
not appear on the right.
\end{proof}
\begin{corollary}\label{cor:normdiff}
For all $f, g\in \mathcal{B}$, we have $\text{N}(g) \leq 2(\text{N}(g+f)+\text{N}(f)).$
\end{corollary}
\begin{proof}
From the lemma, we have $\text{N}(e-f) \leq 2(\text{N}(e)+\text{N}(f))$, which we apply
to $e = g+f$.
\end{proof}
Corollary~\ref{cor:normdiff}, with \eqref{eq:trace} and~\eqref{eq:boundbeta},
now gives
$$
 \N(d^\vee/n)\ \  \leq\ \  2\N(\beta)/\alpha^2 + 2\N(\beta/\alpha + d^\vee/n)
 \ \ \leq\ \  2 (\gamma/\alpha + \N(\beta/\alpha+d^\vee/n))\ \ \leq\ \  2B.
$$
As we also have 
\begin{equation}
\label{eq:n}
n\leq \alpha\gamma\leq \alpha^2 B \leq \frac{1}{4}B^3,
\end{equation}
this gives
$\text{N}(d^\vee) = n^2 \text{N}(d^\vee/n) \leq \frac{1}{8} B^7.$

\newtheorem*{propfield}{Proposition~\ref{prop:field}}

Recall that our goal is to prove the following:
\begin{propfield} 
If $p> \mainbound$, then 
the image $\iota(\mathcal{O})$ is
inside
the ring of $3\times 3$ matrices over a field $\Bone\subset \mathcal{B}$
of degree $\leq 2$ over $\QQ$.
\end{propfield}
\begin{proof}
Suppose $p> \mainbound$.
As $\mu$ generates~$K$, it suffices to show that the entries
$\{x, a, b, c/n, d/n\}$ of $\iota(\mu)$
are in a field~$\mathcal{B}_1$.
Recall that \eqref{eq:identities} gives
$-a=\alpha, \gamma, n\in\ZZ_{>0}$, $b=-\beta$ and
$c = -\frac{n\gamma}{\alpha} - \frac{\beta d}{\alpha}$.
In particular, it suffices to prove that the elements of
$\{x, \beta, d\}$ lie in a field $\mathcal{B}_1$,
for which it suffices to prove that they commute.
We have $\text{N}(x)\leq B$, $\text{N}(\beta)\leq \frac{1}{4} B^3$, $\text{N}(d)\leq \frac{1}{8}B^7$ and $B\geq 2$ (Lemma~\ref{lem:Bnot1}),
hence the product of any pair of distinct elements
of $\{x,\beta,d\}$
has norm less than $p/4$. Therefore, by Lemma \ref{lem:quaternionfield}, every pair
of elements commutes. 
\end{proof}

If $p > \mainbound$, then $\iota(\mu)$ is a matrix
over~$\Bone$.
Let $f$ be the minimal polynomial of $\mu$ over~$\QQ$,
which has degree~$6$.
Then $f(\iota(\mu))=0$, hence $f$ is divisible
by the (at most cubic) minimal polynomial of $\iota(\mu)$ over
the (at most quadratic) field~$\Bone$.
Therefore,
the field $K = \QQ(\mu)$ contains a subfield
isomorphic to $\Bone$ and $\Bone$ is quadratic.
We now identify $\Bone$ with this subfield through a choice of embedding.

This finishes the proof of Theorem~\ref{thm:main}
in the case where $K$ has no imaginary
quadratic subfield.

\section{If the CM field contains an imaginary quadratic subfield}
\label{sec:primitive}

In this section, we finish the proof of
Theorem~\ref{thm:main}.
By the argument at the end of the previous section,
we are left with the case where $\iota(\mu)$ has entries
in an imaginary quadratic subfield $\Bone$ of~$\mathcal{B}$.
We have identified $\Bone$ with the subfield $K_1\subset K$ through a choice of embedding.

Let $p > \mainbound$ be a prime where $B$ is as in Section~\ref{sec:coeffbound}.
Recall that we have a curve $C$ over a number field~$M$ and
a prime $\mathfrak{p}\mid p$ of~$M$
such that $J = \Jacof{C}$ has good reduction at~$\mathfrak{p}$,
but $C$ does not have potential good reduction
at~$\mathfrak{p}$.
By extending $M$ if necessary, assume without loss of generality that $M$ contains the images of all
embeddings
$K\hookrightarrow \overline{M}$.

Recall that the CM type is primitive,
hence is not induced by a CM type of~$\Bone=K_1\subset K$. This means that the CM type induces two distinct embeddings of $\Bone$ into $M$. This primitivity will play a crucial role in our proof of Theorem~\ref{thm:main}. We will need to be able to distinguish between the two embeddings in characteristic~$p$,
for which we will use an element $\sqrt{-\delta}\in\mathcal{O}$
with $\delta\in\ZZ_{>0}$ and $p\nmid 2\delta$.
Such an element automatically exists if $p\nmid 2\Delta(\mathcal{O})$,
which is a relatively weak condition to have in a result
like Theorem~\ref{thm:main}.
However, we do not even need to add such a condition
to the theorem because of the following lemma.

\begin{lemma}
\label{lem:distinctmodp}
Let $B = -\frac{1}{2}\Tr_{K/\mathbb{Q}}(\mu^2)$ and suppose that $p>\frac{1}{4} B^{7.5}$. Then there exists a $\delta\in\mathbb{Z}_{>0}$ coprime to $p$ such that $\sqrt{-\delta}\in\mathcal{O}\cap \Bone$.
\end{lemma}
\begin{proof}
We will prove the lemma with $\delta=-\Delta(\mathcal{O}\cap \Bone)$. Then $\sqrt{-\delta}\in \mathcal{O}\cap \Bone$ and $\delta\in \mathbb{Z}_{>0}$ since $\Bone$ is imaginary quadratic. We must show that $\delta$ is coprime to $p$. Note that 
$\Delta(\mathcal{O}\cap \Bone)=[\mathcal{O}_{\Bone}:\mathcal{O}\cap \Bone]^2\Delta(\mathcal{O}_{\Bone})$.
So it will suffice to prove that both $[\mathcal{O}_{\Bone}:\mathcal{O}\cap \Bone]$ and $\Delta(\mathcal{O}_{\mathcal{B}_1})$ are coprime to~$p$,
which we do by showing that they are smaller than $\frac{1}{4} B^{7.5}$ in absolute value.

Let $a\geq b\geq c\geq 0$ be such
that the images of $\mu$ for the embeddings
$K\rightarrow \CC$ are $\{\pm ai, \pm bi,\pm ci\}$,
so $B = a^2+b^2+c^2$. We have 
\begin{align*}
[\mathcal{O}:\mathbb{Z}[\mu]]^2[\mathcal{O}_K:\mathcal{O}]^2|\Delta(\mathcal{O}_{K})|
 &=|\Delta(\mathbb{Z}[\mu])|=(2a)^2(2b)^2(2c)^2(a-b)^4(a+b)^4(a-c)^4(a+c)^4(b-c)^4(b+c)^4  \\
&= 2^6 a^2b^2c^2(a^2-b^2)^4(a^2-c^2)^4(b^2-c^2)^4,
\end{align*}
which, by the inequality of arithmetic and geometric
means, is less than or equal to
\begin{equation}
 2^6 \left(\frac{a^2+b^2+c^2}{3}\right)^3 
\left(\frac{a^2-b^2+a^2-c^2+b^2-c^2}{3}\right)^{12} \leq 2^{6+12} 3^{-(3+12)} B^{15} < 0.019 B^{15}.\label{eq:disc2}
\end{equation}
Since $\frac{\mathcal{O}_{\Bone}}{\mathcal{O}\cap \Bone} \hookrightarrow \frac{\mathcal{O}_K}{\mathcal{O}}$, by \eqref{eq:disc2} we get 
$[\mathcal{O}_{\Bone}:\mathcal{O}\cap \Bone]^2\leq [\mathcal{O}_K:\mathcal{O}]^2<0.019B^{15}$
which gives $[\mathcal{O}_{\Bone}:\mathcal{O}\cap \Bone]<0.14B^{7.5}$, as desired.
Now for $\Delta(\mathcal{O}_{\mathcal{B}_1})$, we use the tower law for discriminants and \eqref{eq:disc2} to get
$$|\Delta(\mathcal{O}_{\mathcal{B}_1})|^3\leq |\Delta(\mathcal{O}_{\mathcal{B}_1})^3 N_{\Bone/\mathbb{Q}}(\Delta_{K/\Bone})|=|\Delta(\mathcal{O}_K)|< 0.019B^{15}. $$
Hence $|\Delta(\mathcal{O}_{\mathcal{B}_1})|<0.27B^5<\frac{1}{4}B^{7.5}$ (by Lemma~\ref{lem:Bnot1}) and our proof is complete.
\end{proof}

\subsection{Some facts about tangent spaces} In order to detect the CM type (and its all-important primitivity), we use the tangent space to $J=\Jacof{C}$ at the identity. For our discussion, we collect some necessary notions about tangent spaces.
We use the definition of tangent space as given
by Demazure in Expos\'e~II of SGA~3~\cite{SGA3ExposeII} in the special case of a scheme over an affine base scheme. This requires the use of the ring of dual numbers. 
\begin{definition}
 For any commutative ring $R$, let $R[\epsilon]$ denote the $R$-algebra of dual numbers over $R$. It is free with basis $1, \epsilon$ as an $R$-module and the $R$-algebra structure comes from setting $\epsilon^2=0.$  
\end{definition}
The natural inclusion $R\hookrightarrow R[\epsilon]$ induces the structure morphism $\rho:\Spec(R[\epsilon])\rightarrow \Spec(R)$. The natural map $R[\epsilon]\rightarrow R$ which sends $\epsilon \mapsto 0$ induces a section $\sigma:\Spec(R)\rightarrow \Spec(R[\epsilon])$, called the zero section.

Let $X\rightarrow S=\Spec(R)$ be a morphism of schemes and let $u\in X(S)=\Hom_S(S,X)$.  
In \cite{SGA3ExposeII}, Demazure defines a commutative $S$-group scheme called the tangent space of $X/S$ at $u$. We will denote the tangent space of $X/S$ at $u$ by $T^u_{X/S}$. For
a commutative $R$-algebra $R'$,
let $t:\Spec(R')\rightarrow\Spec(R)$
denote the structure morphism.
The set $T^u_{X/S}(R')$ is
defined to be
the collection of $S$-morphisms $\theta:\Spec(R'[\epsilon])\rightarrow X$
making the following diagram commute:
$$
\xymatrix{\Spec(R'[\epsilon])\ar[r]^{\theta}& X\\
\Spec(R')\ar[r]_t \ar[u]_{\sigma}& \Spec(R)\ar[u]^u}
$$

We now gather some general facts about tangent spaces that we will need in our discussion.

\begin{proposition}
\label{prop:module}
The set $T^u_{X/S}(R')$ has a canonical $R'$-module structure. The zero element is the map $u\circ t\circ \rho$ where $\rho: \Spec(R'[\epsilon])\to \Spec(R')$ denotes the structure morphism.
\end{proposition}

\begin{proof}
This is a slight generalization of  
\href{http://stacks.math.columbia.edu/tag/0B2B}{the lemma with tag 0B2B} in
the Stacks Project
\cite{stacks-project}. The proof is the same; we recall the main ingredients here for the reader's convenience. We have a pushout in the
category of schemes
\[
\Spec(R'[\epsilon]) \amalg_{\Spec(R')} \Spec(R'[\epsilon])
= \Spec(R'[\epsilon_1, \epsilon_2])
\]
where $R'[\epsilon_1, \epsilon_2]$ is the $R'$-algebra with
basis $1, \epsilon_1, \epsilon_2$ and
$\epsilon_1^2 = \epsilon_1\epsilon_2 = \epsilon_2^2 = 0$.
Given two $S$-morphisms
$\theta_1, \theta_2 : \Spec(R'[\epsilon]) \to X$,
we construct an $S$-morphism
\begin{equation} \label{eq: addition}
\theta_1 + \theta_2 \ \ :\ \ 
\Spec(R'[\epsilon]) \xrightarrow{\phantom{\theta_1, \theta_2}}
\Spec(R'[\epsilon_1, \epsilon_2])
\xrightarrow{\theta_1, \theta_2} X
\end{equation}
where the first arrow is given by $\epsilon_i \mapsto \epsilon$.
Now for scalar multiplication, given $\lambda \in R'$ there is a selfmap
of $\Spec(R'[\epsilon])$ corresponding to the $R'$-algebra
endomorphism of $R'[\epsilon]$ which sends $\epsilon$ to $\lambda \epsilon$.
Precomposing $\theta : \Spec(R'[\epsilon]) \to X$
with this selfmap gives $\lambda \cdot \theta$. The axioms of a vector space are verified by exhibiting suitable commutative diagrams of schemes. The statement about the zero element follows immediately from the description of the addition law (\ref{eq: addition}).
\end{proof}

\begin{proposition}
\label{prop:link}
Let $v$ be the composition $v: \Spec(R')\xrightarrow{t} S\xrightarrow{u} X$.
Then there is an isomorphism of $R'$-modules 
$T^u_{X/S}(R')\cong \Hom_{R'}(v^*(\Omega^1_{X/S}), R'),$
where $\Omega^1_{X/S}$ denotes the sheaf of relative differentials of $X/S$.
\end{proposition}

\begin{proof}
See Remark 3.6.1 and footnote (25) in \cite{SGA3ExposeII}.
\end{proof}

\begin{proposition}
\label{prop:linear}
Let $X$ and $Y$ be schemes over $S$ and let $f:X\to Y$ be an $S$-morphism. Then $f$ induces an $S$-morphism $T(f):T^u_{X/S}\to T^{f\circ u}_{Y/S}$, called the derived morphism, with the following properties:
\begin{enumerate}
\item $T(f\circ g)=T(f)\circ T(g)$;
\item $T(f)$ induces an $R'$-module homomorphism $T^u_{X/S}(R')\to T^{f\circ u}_{Y/S}(R').$
\end{enumerate}
Furthermore, suppose that $G$ is a group scheme over $S$ with identity section $e$ and $n_G : G \to G$ is the $S$-morphism $g\to g^n$ for $n\in\mathbb{Z}$.
Then the derived morphism $T(n_G) : T^e_{G/S}\to T^e_{G/S}$ is multiplication by $n$, meaning it sends $x\in  T^e_{G/S}(R')$ to $nx$.
\end{proposition}

\begin{proof}
See \cite{SGA3ExposeII}, Proposition 3.7.bis and Corollaire 3.9.4. If $\theta : \Spec(R'[\epsilon])\to X$ is an element of $T^u_{X/S}(R')$, then $T(f)$ sends $\theta$ to $f\circ \theta$. This clearly preserves the $R'$-module structure described in Proposition \ref{prop:module}.
\end{proof}

\begin{proposition}[Proposition 3.8 in \cite{SGA3ExposeII}]
\label{prop:products}
Let $X$ and $Y$ be schemes over $S$. Then 
$T^u_{X/S}\times_S T^w_{Y/S}\cong T^{(u,w)}_{(X\times_S Y)/S}.$
\end{proposition}

\subsection{The proof of Theorem~\ref{thm:main}} \label{sec:proofmain}
Now we will apply these general facts about tangent spaces to our specific case. We want to relate the tangent space of $J$ at the identity to the tangent space of its reduction modulo $\mathfrak{p}$ at the identity. For this we will use the tangent space at the identity section of a N\'{e}ron model of $J/M$.

Let $\mathcal{O}_{\frp}$ be the valuation ring of $\frp$ and let $\mathfrak{K}=\mathcal{O}_M/\mathfrak{p}$ be the residue field.
Let $\mathcal{J}/\mathcal{O}_{\mathfrak{p}}$ be a N\'{e}ron model for $J/M$ and let $\overline{J}/\mathfrak{K}$ be the
special fibre of $\mathcal{J}$. Let $\tilde{e}:\Spec(\mathcal{O}_{\frp})\rightarrow \mathcal{J}$, $e:\Spec (M)\rightarrow J$ and $e_0: \Spec(\resfield)\rightarrow \overline{J}$ be the identity sections of $\mathcal{J}$, $J$ and $\overline{J}$ respectively. 

\begin{lemma}
\label{lem:tensors}
The $\mathcal{O}_{\mathfrak{p}}$-module $T^{\tilde{e}}_{\mathcal{J}/\mathcal{O}_{\mathfrak{p}}}(\mathcal{O}_{\mathfrak{p}})$ is free of rank~$3$. Furthermore, there are natural isomorphisms
\begin{equation}\label{eq:tensork} 
T^e_{J/M}(M)
\cong T^{\tilde{e}}_{\mathcal{J}/\mathcal{O}_{\mathfrak{p}}}(\mathcal{O}_{\mathfrak{p}})\otimes_{\mathcal{O}_{\mathfrak{p}}}M \qquad \text{and}  \qquad
T^{e_0}_{\overline{J}/\mathfrak{K}}(\mathfrak{K})
\cong T^{\tilde{e}}_{\mathcal{J}/\mathcal{O}_{\mathfrak{p}}}(\mathcal{O}_{\mathfrak{p}})\otimes_{\mathcal{O}_{\mathfrak{p}}}\mathfrak{K}
\end{equation} 
as vector spaces over $M$ and $\mathfrak{K}$ respectively.
Moreover, the isomorphisms \eqref{eq:tensork} respect the
action of $T(f)$ for $f\in\End_{\mathcal{O}_\mathfrak{p}}(\mathcal{J}) = \End_M(J)$.
\end{lemma}

\begin{proof}
By \cite[Proposition 6.2.5]{Liu}, $\Omega^1_{\mathcal{J}/\mathcal{O}_{\mathfrak{p}}}$ is free of rank~$3$ in a neighborhood of the image of $e_0$. Note that any such neighborhood contains the image of $\tilde{e}$. Therefore, $\tilde{e}^*(\Omega^1_{\mathcal{J}/\mathcal{O}_{\mathfrak{p}}})$ is a free $\mathcal{O}_{\mathfrak{p}}$-module of rank~$3$. Now Proposition \ref{prop:link} implies that the same is true of $T^{\tilde{e}}_{\mathcal{J}/\mathcal{O}_{\mathfrak{p}}}(\mathcal{O}_{\mathfrak{p}})$. Likewise, 
$T^e_{J/M}(M)$ and $T^{e_0}_{\overline{J}/\mathfrak{K}}(\mathfrak{K})$ are vector spaces of dimension $3$ over $M$ and $\mathfrak{K}$, respectively.

We have canonical identifications $T^e_{J/M}(M)=T^{\tilde{e}}_{\mathcal{J}/\mathcal{O}_{\mathfrak{p}}}(M)$ and $T^{e_0}_{\overline{J}/\mathfrak{K}}(\mathfrak{K})=T^{\tilde{e}}_{\mathcal{J}/\mathcal{O}_{\mathfrak{p}}}(\mathfrak{K})$. Let $F\in\{M,\mathfrak{K}\}$ and let $t:\Spec(F[\epsilon])\rightarrow \Spec(\mathcal{O}_\mathfrak{p}[\epsilon])$ be the natural map. Then precomposing an element $\theta\in T^{\tilde{e}}_{\mathcal{J}/\mathcal{O}_{\mathfrak{p}}}(\mathcal{O}_\mathfrak{p})$ with $t$ gives an element of $T^{\tilde{e}}_{\mathcal{J}/\mathcal{O}_{\mathfrak{p}}}(F)$. The $\mathcal{O}_\mathfrak{p}$-bilinear map $T^{\tilde{e}}_{\mathcal{J}/\mathcal{O}_{\mathfrak{p}}}(\mathcal{O}_\mathfrak{p})\times F\to T^{\tilde{e}}_{\mathcal{J}/\mathcal{O}_{\mathfrak{p}}}(F)$ given by $(\theta,\lambda)\mapsto \lambda\cdot (\theta\circ t)$ induces a homomorphism of $F$-vector spaces $T^{\tilde{e}}_{\mathcal{J}/\mathcal{O}_{\mathfrak{p}}}(\mathcal{O}_\mathfrak{p})\otimes_{\mathcal{O}_\mathfrak{p}} F\to T^{\tilde{e}}_{\mathcal{J}/\mathcal{O}_{\mathfrak{p}}}(F).$
One can check that this map is injective using the description of the zero element given in Proposition~\ref{prop:module}. Surjectivity follows by comparing dimensions. Finally, the action of $T(f)$, as described in Proposition~\ref{prop:linear}, is clearly preserved.
\end{proof}

The main ingredient for the proof of Theorem~\ref{thm:main} is the following proposition.

\begin{proposition}
\label{prop:2evals}
Let $B$ and $\delta$ be as in Lemma \ref{lem:distinctmodp}, and suppose that $p > \frac{1}{8}B^{10}$. 
Then there is an invertible matrix $P\in \Mat{3}(\Bone)$ such that
\[P\iota(\sqrt{-\delta}) P^{-1}=\pm\begin{pmatrix} \sqrt{-\delta} & 0 & 0\\
0 & \sqrt{-\delta} & 0\\
0 & 0 & -\sqrt{-\delta}\end{pmatrix}.\]
\end{proposition}

\begin{proof}
By Proposition \ref{prop:field}, reduction at a prime above $p>\mainbound$ induces a $\mathbb{Q}$-algebra homomorphism $\iota: K\hookrightarrow \End(E^3)\otimes \mathbb{Q}=\text{Mat}_3(\mathcal{B})$ with image contained in $\text{Mat}_3(\Bone)$, the ring of $3\times 3$ matrices over $\Bone$. 
Since $\iota(\sqrt{-\delta})^2=-\delta I_3$ and $\sqrt{-\delta}\in \Bone$,
we can take a change of basis over $\Bone$ to diagonalize the matrix
$\iota(\sqrt{-\delta})$. Moreover, the eigenvalues of $\iota(\sqrt{-\delta})$ are in $\{\pm\sqrt{-\delta}\}$. It suffices to show that $\iota(\sqrt{-\delta})$ has two distinct eigenvalues, i.e.   $\iota(\sqrt{-\delta})\neq \pm \sqrt{-\delta}I_3$. For this we will use the primitivity of the CM type. In order to detect the CM type, we will use the tangent space to $J=\Jacof{C}$ at the identity. 

By the N\'{e}ron mapping property, $\sqrt{-\delta}$ has a unique extension to an $\mathcal{O}_\mathfrak{p}$-endomorphism of the N\'{e}ron model $\mathcal{J}$, which we will denote by $\varphi$. 
Applying Proposition \ref{prop:linear}, we get an endomorphism $T(\varphi)$ of $T^{\tilde{e}}_{\mathcal{J}/\mathcal{O}_{\mathfrak{p}}}$
which induces an $\mathcal{O}_{\mathfrak{p}}$-linear endomorphism of  $T^{\tilde{e}}_{\mathcal{J}/\mathcal{O}_{\mathfrak{p}}}(\mathcal{O}_{\mathfrak{p}})$.
By the definition of primitivity of the CM type
(Definitions \ref{def:primitivetype} and~\ref{def:CMtypeof}),
the action of $T(\varphi)$ on $T^e_{J/M}(M)$ has two distinct eigenvalues, $\sqrt{-\delta}$ and $-\sqrt{-\delta}\in\OMp$.
By Lemmas~\ref{lem:Bnot1} and \ref{lem:distinctmodp}, the two eigenvalues $\pm\sqrt{-\delta}$ remain distinct in the residue field $\mathfrak{K}$, which has characteristic $p>\mainbound\geq \frac{1}{4}B^{7.5}>2$. Therefore, applying Lemma \ref{lem:tensors} again, we see that the action of  $T(\varphi)$ on $T^{e_0}_{\overline{J}/\mathfrak{K}}(\mathfrak{K})$ has two distinct eigenvalues. 

Now, let the isogeny $F: E^3\rightarrow \overline{J}$ be as in Sections \ref{sec:embprb} and~\ref{sec:coeffbound}.
By Lemma~\ref{lem:pol_matrix},
there is an integer $n>0$ and an isogeny $G: \overline{J}\rightarrow E^3$
such that $GF$ is multiplication by $n$
on~$E^3$. Then $\iota(\sqrt{-\delta})=n^{-1}G\overline{\varphi}F$, where $\overline{\varphi}$ denotes the $\mathfrak{K}$-endomorphism of $\overline{J}$ induced by $\varphi$. Recall from \eqref{eq:n} in Section \ref{sec:coeffbound} that $0<n\leq \frac{B^3}{4}<\mainbound<p$. Therefore, $n$ is invertible in $\mathfrak{K}$ and Proposition \ref{prop:linear} gives
\[T(G\overline{\varphi} F) = T(G)\circ T(\varphi)\circ T(F) = n  n^{-1}T(G)\circ T(\varphi)\circ T(F) = n  T(F)^{-1} \circ T(\varphi)\circ T(F).\]
The right-hand side is $n$ times a conjugate of $T(\varphi)$, whereby its eigenvalues
in~$\overline{\resfield}$ are $n$ times those of $T(\varphi)$. Therefore, $T(G\overline{\varphi} F)$ has two distinct eigenvalues in~$\resfield$ for its action on the tangent space $T^0_{E^3/\mathfrak{K}}(\mathfrak{K})$ of $E^3$ at the identity.
By Proposition \ref{prop:products}, we have
$T_{E^3/\mathfrak{K}}^0(\mathfrak{K})\cong T_{E/\mathfrak{K}}^0(\mathfrak{K})\oplus T_{E/\mathfrak{K}}^0(\mathfrak{K})\oplus T_{E/\mathfrak{K}}^0(\mathfrak{K})$
where $T_{E/\mathfrak{K}}^0(\mathfrak{K})$ denotes the tangent space of $E$ at the identity.
Now suppose that $n\iota(\sqrt{-\delta})=\pm n\sqrt{-\delta}I_3$.
Then $T(G\overline{\varphi} F) = T(n\iota(\sqrt{-\delta})) =
\pm n \sqrt{-\delta}I_{3}$ has only one eigenvalue.
Contradiction.
\end{proof}

\begin{proof}[Proof of Theorem~\ref{thm:main}.]
Suppose that $p > \mainbound$. Recall from the end of Section~\ref{sec:coeffbound}
that $\iota(\mu)$ has coefficients in a quadratic field
$\Bone\ni \sqrt{-\delta}$.
Applying Proposition \ref{prop:2evals}, we see that since $\mu$ commutes with $\sqrt{-\delta}$, the matrix
$P\iota(\mu)P^{-1}$ is of the form
\[\begin{pmatrix} * & * & 0\\
*& * & 0\\
0 & 0 & *\end{pmatrix}.\]
But this means that the bottom right entry of $P\iota(\mu)P^{-1}$ is a root of the (irreducible degree six) minimal polynomial of $\mu$ over $\mathbb{Q}$. This gives a contradiction because the entries of the matrix $P\iota(\mu)P^{-1}$ lie in the quadratic field $\Bone$. This completes the proof of Theorem~\ref{thm:main}.
\end{proof}

\section{Geometry of numbers}\label{sec:geom}

The following is a reformulation and proof of Proposition~\ref{prop:geomnumbers}.

\begin{proposition}
Let $\mathcal{O}$ be an order in a sextic CM field $K$ with totally real cubic subfield $K_+$.
\begin{enumerate}
 \item \label{case1} If $K$ contains no imaginary quadratic subfield, then there exists \mbox{$\mu\in\OO$} such that $K = \QQ(\mu)$ and $\mu^2$ is a totally negative element in $K_+$ with $0<-\mathrm{Tr}_{K_+/\QQ}(\mu^2)\leq (\frac{6}{\pi})^{2/3}|\Delta(\OO)|^{1/3}$. 
\item \label{case2} If $K$ contains an imaginary quadratic subfield~$K_1$,
we write $\mathcal{O}_i = K_i\cap \mathcal{O}$ for $i\in\{1,+\}$.
Then $\exists$ \mbox{$\mu\in \OO$}
such that $K = \QQ(\mu)$ and $\mu^2$ is a totally negative element in 
$K_+$ satisfying \[0 < -\Tr_{K_+/\mathbb{Q}}(\mu^2) \leq |\Delta(\OO_1)|(1+2\sqrt{|\Delta(\OO_+)|}).\]
\end{enumerate}
\end{proposition}
\begin{proof}
Let $\Phi=\{\phi_1,\phi_2,\phi_3\}$ be the set of embeddings of $K$ into $\CC$ up to complex conjugation. 
We identify $K\otimes_{\QQ}\RR$ with $\CC^3$ via the $\RR$-algebra isomorphism $K\otimes_\QQ\RR\rightarrow\CC^3: x\otimes a \mapsto (a\phi_1(x), a\phi_2(x),a\phi_3(x))$.

\noindent \textit{(1)}
The order $\OO\subset K$ is a lattice of co-volume $2^{-3}|\Delta(\OO)|^{1/2}$ in $\CC^3$.
We define the symmetric  convex body $$
\calC_R = \{x=(x_1,x_2,x_3)\in \CC^3: |\Real(x_i)|<1\ \text{for all $i$, }\ \sum_{i}\im(x_i)^2<R^2\}  \subset \CC^3.
$$

Next, we claim that if $R=(\frac{3}{4\pi})^{1/3}|\Delta(\OO)|^{1/6}+\epsilon$ for some $\epsilon>0$, then there is a non-zero $\gamma\in\OO\cap \calC_R$ such that $K = \QQ(\gamma)$.
Indeed, suppose that $R=(\frac{3}{4\pi})^{1/3}|\Delta(\OO)|^{1/6}+\epsilon$.
Then we have 
$$\vol(\calC_R) = 2^3(\frac{4}{3}\pi R^3)>2^3|\Delta(\OO)|^{1/2} = 2^6\covol(\OO).$$
By Minkowski's first convex body theorem (see Siegel~\cite[Theorem 10]{Siegel}),
there is a non-zero element $\gamma$ in $\OO\cap\calC_R$.
If $\gamma\in K_+$, then we have
$|\N_{K_+/\QQ}(\gamma)| = \prod_{\phi_i\in\Phi}|\Real(\phi_i(\gamma))|<1$,
so we get $\gamma=0$, a contradiction. Hence $\gamma\in\OO\cap \calC_R$ and $\gamma\notin K_+$.
To prove the claim, it only remains to prove
that $\gamma$ generates~$K$. As $K$ has degree~$6$,
the field generated by $\gamma$ has degree $1$, $2$, $3$ or $6$.
Since any subfield of a CM field is either totally real or a CM field, we find that either $\gamma$
is totally real (hence in $K^+$, contradiction),
or generates a CM subfield of $K$.
As CM fields have even degree and
we are in case \eqref{case1}, where $K$ has no imaginary
quadratic subfield, we find that $\mathbb{Q}(\gamma)=K$.
This proves the claim.

Let $\mu = \gamma-\overline{\gamma}$. Then $\mu^2$ is a totally negative element in $K_+$, hence $\mathbb{Q}(\mu)=K$. We get 
$$
-\Tr_{K_+/\QQ}(\mu^2)=-\sum_{i}\phi_i(\mu^2) = -\sum_{\phi_i\in\Phi}\phi_i((\gamma-\overline{\gamma})^2)=4\sum_{i}\im(\phi_i(\gamma))^2< 4R^2.
$$
Since $\gamma$ is an algebraic integer in~$K$, we have $\Tr_{K_+/\QQ}(\mu^2)\in\ZZ$. So when
we let $\epsilon$ tend to $0$, we get $-\Tr_{K_+/\QQ}(\mu^2)\leq 4(\frac{3}{4\pi})^{2/3}|\Delta(\OO)|^{1/3}= (\frac{6}{\pi})^{2/3}|\Delta(\OO)|^{1/3}$, which proves \eqref{case1}.

\noindent \textit{(2)}
The order $\OO_+\subset K_+$ is a lattice of co-volume $|\Delta(\OO_+)|^{1/2}$
in $\RR^3$. We define the symmetric convex body
$$
\calC_R = \{x=(x_1,x_2,x_3)\in \RR^3: |x_1| < 1, |x_2| < R, |x_3| < R\}\subset \RR^3.
$$

Next, we claim that if $R=|\Delta(\OO_+)|^{1/4}+\epsilon$ for some $\epsilon>0$, then there is a non-zero $\gamma\in\OO_+\cap \calC_R$ such that $\gamma\in K_+\setminus \QQ$.
Indeed, we then have $\vol(\calC_R) = 2^3R^2>2^3|\Delta(\OO_+)|^{1/2} = 2^3\covol(\OO_+)$.
By Minkowski's first convex body theorem (see Siegel~\cite[Theorem 10]{Siegel}),
there is a non-zero element $\gamma$ in $\OO_+\cap\calC_R$.
If $\gamma\in \QQ$, then $\gamma\in\ZZ$, but $|\gamma|< 1$,
so we get $\gamma=0$, a contradiction. Hence $\gamma\in\OO_+\cap \calC_R$ and $\gamma\notin \QQ$.
This proves the claim.

Let $\mu = \sqrt{\Delta(\OO_1)}\gamma$. Then $\mu^2$ is a totally negative element in $K_+$. We get 
$
-\Tr_{K_+/\QQ}(\mu^2)=|\Delta(\OO_1)| \sum_{i}\phi_i(\gamma^2) \leq 
|\Delta(\OO_1)| (1+2R^2).
$
Since $\gamma$ is an algebraic integer in~$K_+$, we have $\Tr_{K_+/\QQ}(\mu^2)\in\ZZ$. So when
we let $\epsilon$ tend to~$0$, we get $-\Tr_{K_+/\QQ}(\mu^2)\leq |\Delta(\OO_1)|(1+2|\Delta(\OO_+)|^{1/2})$, as desired.
 \qedhere
\end{proof}

\section{Invariants}\label{sec:invariants}

In this section, we prove Theorem \ref{thm:invariantdenom}. 
Let $j=\frac{u}{\Delta^l}$ be as in that theorem: a quotient of invariants 
of hyperelliptic (respectively Picard) curves $y^2=F(x,1)$ (respectively $y^3=f(x)$) of genus~$3$, where~$\Delta$ is the 
discriminant of~$F$ (respectively $f$) and $u$ is an invariant of weight $56l$ (respectively $12l$).
Let $C$ be such a curve, not necessarily with CM, over a number field~$M$.

\begin{theorem}\label{thm:stablebad}
In the situation above, if $j(C)$ has negative valuation at a prime $\mathfrak{p}$ with $\mathfrak{p}\nmid 6$,
then $C$ does not have potential good reduction at~$\mathfrak{p}$.
\end{theorem}

\begin{proposition}[Example 10.1.26 in \cite{Liu}]\label{prop:hyper} Let $S=\Spec (\OMp)$, where $\OMp$ is a discrete valuation ring with field of fractions $M$ and residue field $\resfield$
with $\text{char}(\resfield)\neq 2$. Let $C$ be a hyperelliptic curve of genus $g\geq1$ over~$M$ defined by an affine equation
$
y^2=P(x),
$
with $P(x)\in M[x]$ separable. Then $C$ has good reduction if and only if 
$C$ is isomorphic to a curve given
by an equation as above with $P(x)\in\OMp[x]$ such that the image
of $P(x)$ in $\resfield[x]$ is separable of degree $2g+1$ or $2g+2$. Furthermore, any such isomorphism is given by a change of variables as in \cite[Corollary 7.4.33]{Liu}.\qed
\end{proposition}

\begin{proof}[Proof of Theorem~\ref{thm:stablebad}]
For the Picard case: let $C$ be a Picard curve of potentially good reduction at~$\mathfrak{p}$,
given by an affine equation as in~(\ref{eq:picard}).
Corollary~3.20 in Lercier, Liu, Lorenzo Garc\'ia and Ritzenthaler~\cite{LLLR}
says that
$v(a_2)\geq \frac{2}{12} v(\Delta)$ and
$v(a_4)\geq \frac{4}{3} v(\Delta)$,
where $v$ is the $\mathfrak{p}$-adic valuation.
In order to prove $v(j(C))\geq 0$, it now suffices
to prove $v(a_3)\geq \frac{3}{12} v(\Delta)$.
So suppose $v(a_3) = \frac{3}{12} v(\Delta)-e$ with $e>0$.
Writing out the discriminant gives
$
a_{3}^{4} + 3^{-3}(4 a_{2}^{3} -144 a_2a_4) a_{3}^{2}  + 3^{-3}(-16 a_{2}^{4} a_{4} + 128 a_{2}^{2} a_{4}^{2} - 256 a_{4}^{3}+\Delta)=0,
$
but the term $a_3^4$ has strictly lower valuation than all other terms,
which is impossible, hence $v(j(C))\geq 0$.

For the hyperelliptic case:
assume that $C$ has potential good reduction at $\mathfrak{p}$ with $\mathfrak{p}\nmid 6$. Extend the base field so that $C$ has good reduction, and then take a model $y^2=P(x)\in\OMp[x]$ such that the image
of $P(x)$ in $\resfield[x]$ is separable of degree $2g+1$ or $2g+2$, as in
Proposition~\ref{prop:hyper}. This changes the coefficients, but not the
normalized invariant $j = \frac{u}{\Delta^l}$ by the definition of hyperelliptic curve invariants. 
Since $P(x)$ has coefficients in $\OMp$, it follows that $u(P(x))\in\OMp$. Also, we have $\Delta(P(x))\in\OMp^*$ by 
Proposition~\ref{prop:hyper}, hence $j(C)\in\OMp$. This contradicts the assumption and, therefore, the theorem follows.
\end{proof}

\begin{remark}
	Results of Bouw, Koutsianas, Sijsling and Wewers \cite{BKSW} give an alternative proof
	of Theorem~\ref{thm:stablebad} for Picard curves as follows.
	Numbered results referenced below are from~\cite{BKSW};
	some assume $\mathfrak{p}\nmid 3$.
	
	First of all, \cite{BKSW} distinguishes between ``special'' Picard curves
	($\overline{M}$-isomorphic to $S:y^3 = x^4 - 1$, Lemma~1.17)
	and ``non-special'' Picard curves (the rest).
	We compute $\Delta(S)=-2^8$, so $v(\Delta(S))=0$, hence
	all special
	Picard curves $C$ satisfy $v(j(C))\geq 0$.
	
	Now let $C$ be a non-special Picard curve with potential good reduction.
    Extend the base field to have good reduction.
	Then choose a model of the form $y^3 = cf(x)$ with
	$v(c)\in\{0,1,2\}$ and $f$ monic quartic and ``reduced''
	as in Definitions~3.1.1/4 and Corollary~3.1.18.
	Proposition 3.2.1 gives
	$v(c)=0$
	and $v(\Delta(f))=0$.
	Extend $M$ with $\sqrt[3]{c}$ to get the model $y^2 = f(x)$.
	Then complete the $4$th power to get rid of the $x^3$ coefficient in~$f$.
	Now we have $v(a_2)$, $v(a_3)$, $v(a_4)\geq 0$ and $v(\Delta)=0$, hence $v(j(C))\geq 0$.
\end{remark}

\noindent\textit{Proof of Theorem~\ref{thm:invariantdenom}.}
By Lemma~\ref{lem:Bnot1}, the bound $B$ appearing in Theorem~\ref{thm:main} satisfies $B\geq 2$. Therefore, for $p\leq 3$ we have $p<\mainbound$.
Hence
we assume that $p>3$. 
 In the situation of Theorem~\ref{thm:invariantdenom}, we showed
 in Theorem~\ref{thm:stablebad} that the curve
 does not have potential good reduction,
 hence Theorem~\ref{thm:main} applies.
\qed

\bibliographystyle{plain}
\bibliography{mybib}
\end{document}